%% file: main.tex
\documentclass{article}
\usepackage{graphicx} 
\usepackage{fancyhdr}
\usepackage{extramarks}
\usepackage{amsmath}
\usepackage{amsthm}
\usepackage{amssymb}
\usepackage{amsfonts}
\usepackage{tikz}
\usepackage[plain]{algorithm}
\usepackage{algpseudocode}
\usepackage{relsize}
\usepackage[shortlabels]{enumitem}
\usepackage{url}
\usepackage{hyperref}
\usepackage{appendix}

\usepackage{supertabular}
\usepackage{colortbl}
\usepackage{caption}
\usepackage{float}

\usepackage{tikzscale}
\usetikzlibrary{positioning}

\newcommand{\N}{\mathbb{N}}

\newtheorem{prop}{Proposition}
\newtheorem{cor}{Corollary}
\newtheorem{theorem}{Theorem}

\newtheorem{obs}{Observation}
\newtheorem{defn}{Definition}
\newtheorem{question}{Question}

\usepackage{amsrefs}

\usepackage[vietnamese,english]{babel}
\usepackage[T1,T5]{fontenc}

\begin{document}
	
	\begin{center}
		\vskip 1cm{\Large\bf 
			Constructions of and Bounds on the Toric Mosaic Number
		}\footnote{Subject Classification: 57K10}
		\vskip 1cm{\Large
			Kendall Heiney, Margaret Kipe, Samantha Pezzimenti, Kaelyn Pontes, L\d\uhorn c Ta
		}
	\end{center}
	\vskip .2 in
	
	\begin{abstract}
		Knot mosaics were introduced by Kauffman and Lomonaco in the context of quantum knots, but have since been studied for their own right. A classical knot mosaic is formed on a square grid. In this work, we identify opposite edges of the square to form mosaics on the surface of a torus. We provide two algorithms for efficiently constructing toric mosaics of torus knots, providing upper bounds for the toric mosaic number. Using these results and a computer search, we provide a census of known toric mosaic numbers.
	\end{abstract}
	
	\section{Introduction}
	
	\subsection{Classical knot mosaics}
	
	In 2008, Lomonaco and Kauffman \cite{lomonaco} introduced a method of building knots using the 11 mosaic tiles in Figure \ref{fig:OGtiles} below as a means to describe quantum knots.
	
	\begin{figure}[h]
		\centering
		\includegraphics[scale=0.2]{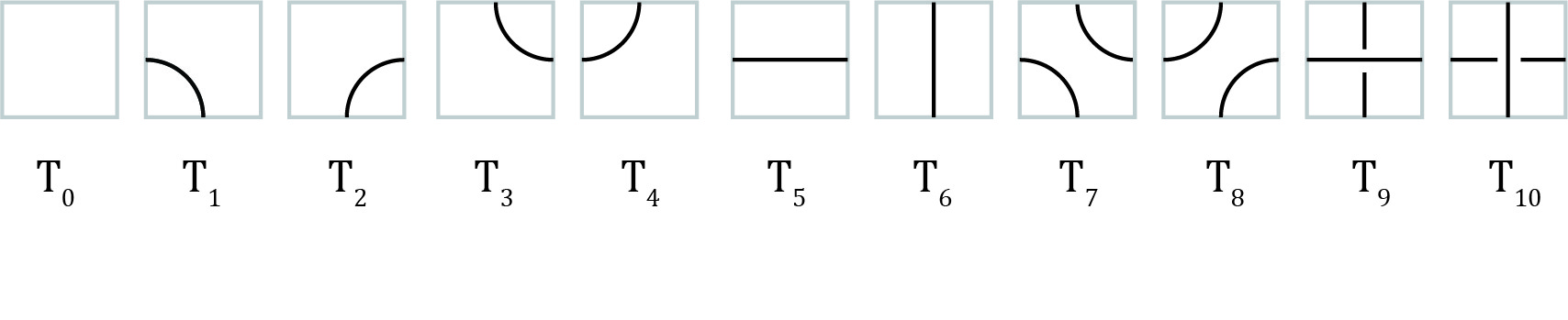}
		\vspace{-30pt}
		\caption{The mosaic tiles $T_0,\ldots,T_{10}$.}
		\label{fig:OGtiles}
	\end{figure}
	
	In particular, an \textbf{$n$-mosaic} is an arrangement of the tiles $T_0,\ldots,T_{10}$ into an $n\times n$ array. Such a mosaic is \textbf{suitably connected} if connection points of each tile coincide with the connection points of adjacent tiles. Notice that such a requirement puts restrictions on allowable tiles in certain positions. For example, tiles $T_7,\ldots, T_{10}$ cannot be placed along the outer boundary. 
	
	Suitably connected $n$-mosaics can be used to create projections of knots and links. For example, the $4$-mosaic in Figure \ref{fig:OGtrefoil} represents a trefoil knot. While knot mosaics can be made arbitrarily large (just add another row and column of $T_0$ tiles), they cannot be made arbitrarily small. The \textbf{mosaic number $m(K)$ of a knot $K$} is the minimum value $n$ such that $K$ can be represented on an $n$-mosaic. The mosaic in Figure \ref{fig:OGtrefoil} shows that $m(3_1)\leq 4$, where $3_1$ denotes the trefoil knot. In fact, since no crossing tiles can appear on the outer boundary, any $3$-mosaic of a knot can only represent the unknot. Thus, $m(3_1)=4.$ The mosaic number has been computed for all knots having up to 10 crossings; see \cite{normal-mosaics}.
	
	\begin{figure}[h]
		\centering
		\includegraphics[scale=0.2]{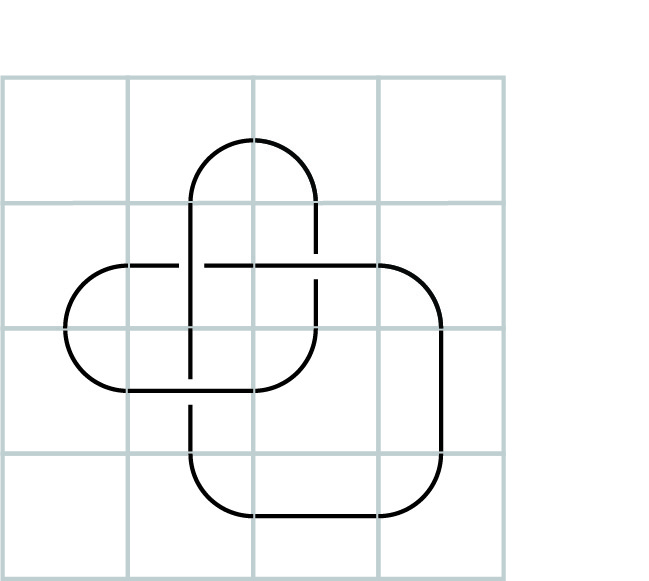}
		\caption{A $4$-mosaic representing a trefoil knot.}
		\label{fig:OGtrefoil}
	\end{figure}
	
	\subsection{Toric knot mosaics}
	
	Motivated by an open question in \cite{cubic}, 
	we now consider mosaics in which opposite edges are identified to form a torus as in Figure \ref{fig:torictrefoil}. Toric link mosaics were introduced in \cite{first-toric-paper} and \cite{second-toric-paper}, where they are referred to as ``toroidal mosaics.'' An $n$-mosaic as defined in the previous section will henceforth be referred to as a \textit{classical $n$-mosaic}. 
	
	\begin{defn}
		A \textbf{toric $n$-mosaic} is a classical $n$-mosaic with opposite edges identified. It is \textbf{suitably connected} if the connection points of all tiles coincide with the connection points of contiguous tiles, and of the tiles on identified edges.
	\end{defn}
	
	Although any embedding can be chosen, we will stick with the convention that first identifies the top and bottom edges, and then identifies the left and right edges. (Note that this is referred to as a ``{longitudinal toroidal mosaic}'' in \cite{first-toric-paper}, and the opposite convention is referred to as a ``{meridianal toroidal mosaic}.'') 
	
	One immediate consequence of this choice of embedding is that crossings are introduced outside of the mosaic grid. We refer to these induced crossings as \textbf{hidden crossings}, and crossings on the mosaic grid as \textbf{visible crossings}. For instance, the toric $2$-mosaic in Figure \ref{fig:torictrefoil} contains only one visible crossing, but has 4 hidden crossings. This particular $2$-mosaic represents a trefoil knot. In general,
	
	\begin{obs}
		A toric $n$-mosaic with $A$ connection points on the left and right edges, and $B$ connection points on the top and bottom edges, will contain $AB$ hidden crossings. 
	\end{obs}
	
	\begin{figure}[h]
		\centering
		\includegraphics[scale=1.0]{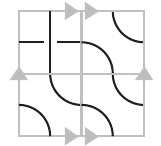}
		\caption{A toric $2$-mosaic representing a trefoil knot.}
		\label{fig:torictrefoil}
	\end{figure}
	
	By design, toric mosaics have the potential to use fewer tiles than classical mosaics to construct knots. We define the following knot invariant:
	
	\begin{defn}
		The \textbf{toric mosaic number $m_T(K)$ of a knot $K$} is the smallest value $n$ such that $K$ can be represented on a toric $n$-mosaic. 
	\end{defn}
	
	Since any classical mosaic can be made into a toric mosaic by identifying edges, the following must hold:
	
	\begin{obs} \label{initial-obs}
		For all knots $K$, $m_T(K)\leq m(K).$
	\end{obs}
	
	For example, if $U$ represents the unknot, $m(U)=2$, while $m_T(U)=1$. 
	Indeed, there are 11 toric $1$-mosaics consisting of only one tile. The toric $1$-mosaics formed from tiles $T_1,\ldots,T_8$ all form unknots. Depending on the embedding chosen, the toric $1$-mosaics formed from tiles $T_9$ and $T_{10}$ could either represent unlinked unknots or Hopf links. Since we are using the longitudinal convention, the toric $1$-mosaic formed by $T_9$ represents a Hopf link, and the toric $1$-mosaic formed by $T_{10}$ represents two unlinked unknots. This shows the following:
	
	\begin{prop}
		$m_T(K)=1$ if and only if $K$ is the unknot. Moreover, $m_T(K)=2$ if and only if $K$ is the trefoil.
	\end{prop}
	
	In \cite{first-toric-paper}, and extended in \cite{second-toric-paper}, the authors produce a catalog of all toric link $2$-mosaics. The only knots that appear on the list are the unknot and trefoil. Thus, the following must hold:
	
	\begin{obs}
		$m_T(K)\geq 3$ for all knots $K$ not equal to the unknot or trefoil.
	\end{obs}
	
	Note that the unknot and trefoil both make the inequality in Observation \ref{initial-obs} strict.

	
	%
	
	\subsection{Torus knots}\label{sec:torus-knots}
	
	A \textbf{torus knot} is a knot that can be embedded on the surface of an unknotted torus in $\mathbb{R}^3$. It can be identified by a pair of coprime integers $(p,q)$ in the following way: the torus knot wraps around the meridian (the short way) $p$ times and around the longitude (the long way) $q$ times.
	
	It is well-known that the crossing number of a $(p,q)$-torus knot is $\min(q(p-1),p(q-1))$ (see, for example \cite{KnotBook}).
	Since $(p,q)$-torus knots and $(q,p)$-torus knots are ambient-isotopic, we will assume without loss of generality that $p<q$.
	Since $(1,q)$-torus knots are unknots and $(-p,-q)$-torus knots are ambient-isotopic to $(p,q)$-torus knots, we will also assume without loss of generality that $p \geq 2$.
	Therefore, the crossing number is simply $q(p-1)$. For more information on torus knots, we direct the reader to Section 5.1 of \cite{KnotBook}.
	
	One \textbf{classical representation} of any $(p,q)$-torus knot on the surface of an unwrapped torus is shown in the top image of Figure \ref{fig:classic-rep}. In this representation, $p$ strands intersect each vertical edge, and $q$ strands intersect each horizontal edge. Moreover, all $q(p-1)$ irreducible crossings occur outside of the diagram. The torus knot represented in Figure \ref{fig:classic-rep} is the $(3,4)$-torus knot, and its $4(3-1)=8$ irreducible hidden crossings are evident in the right image of the figure.
	
	\begin{figure}[h]
		\centering
		\includegraphics[width=0.8\linewidth]{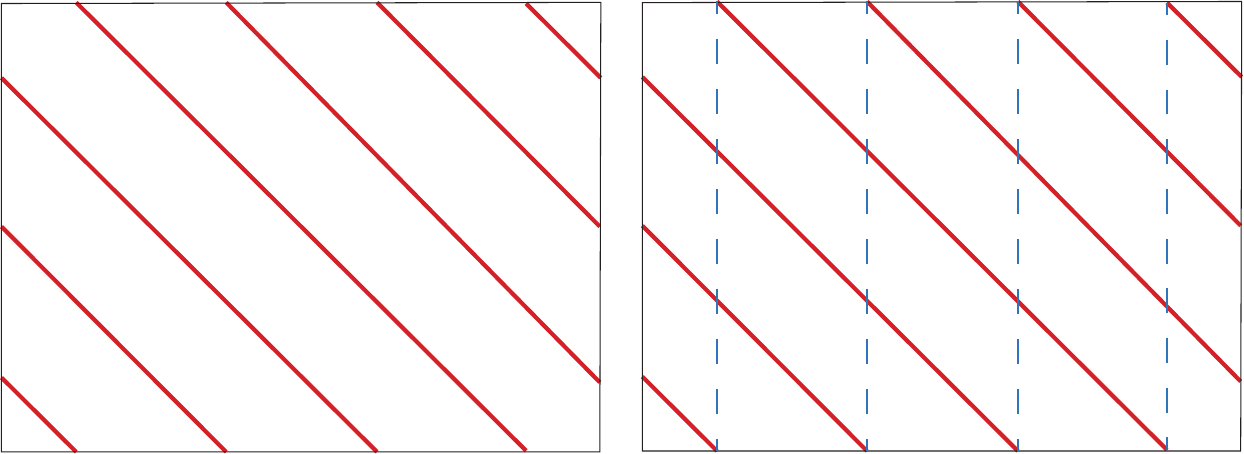}
		\caption{The classical representation of a $(3,4)$-torus knot and its eight irreducible hidden crossings.}
		\label{fig:classic-rep}
	\end{figure}
	
	This representation gives an immediate, albeit naive upper bound for the toric mosaic number of torus knots. 
	Given a $(p,q)$-torus knot, a $p\times q$ grid whose tiles are all $T_7$ gives a non-square mosaic representation of the knot's classical representation. As in Figure \ref{fig:classic-mosaic}, appending a $(q-p)\times q$ grid whose tiles are all $T_6$ to the bottom of this representation gives a toric $q$-mosaic for the knot, giving us the following:
	
	\begin{prop} \label{naivebound}
		If $K$ is a $(p,q)$-torus knot with $2\leq p< q$, then $m_T(K) \leq q$.
	\end{prop}
	
	\begin{figure}[h]
		\centering
		\includegraphics[width=0.3\linewidth]{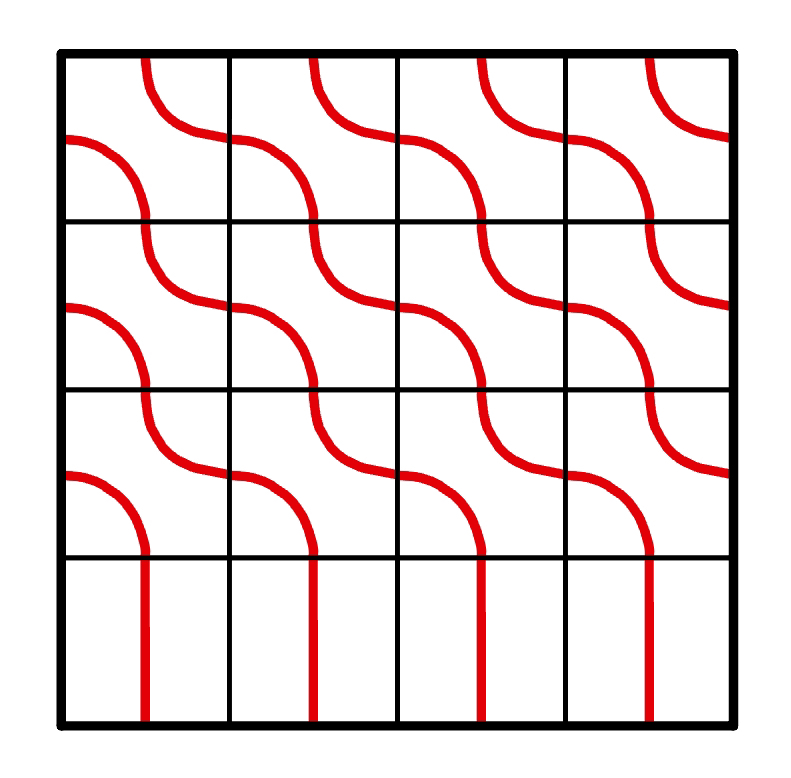}
		\caption{A toric $4$-mosaic representing the $(3,4)$-torus knot inspired by its classical representation.}
		\label{fig:classic-mosaic}
	\end{figure}
	
	In Sections \ref{sec:braid1} and \ref{sec:braid2}, we give algorithms to create more efficient toric mosaics for torus knots. Our main results in these sections are summarized in Theorem \ref{thm:rapunzel} and Corollary \ref{cor:full-braid}.
	
	In Section \ref{sec:census}, we provide a census of toric mosaic numbers for knots with 9 or fewer irreducible crossings using an exhaustive computer search. We describe an outline for the algorithm used, and a link for the full program used. Our main results in this section are summarized in Theorems \ref{thm:alg-1} and \ref{thm:alg-2}.
	
	Finally, in Section \ref{sec:questions}, we suggest some additional areas to explore and open questions related to this work.
	
	\section{Algorithmic upper bounds for $(p,q)$-torus knots} \label{sec:algos}
	
	The construction in Proposition \ref{naivebound} is quite inefficient since the bottom $q-p$ rows only use $T_6$ tiles. In this section, we address this inefficiency by introducing crossings onto the mosaic in a way that decreases the mosaic size and, hence, the upper bound on the toric mosaic number. First, we present a ``one-braid'' algorithm that gives an upper bound on the mosaic number for infinitely many classes of torus knots. Then, we present a ``full-braid'' algorithm that improves upon this bound for infinitely many $(2,q)$-torus knots.
	
	\subsection{One-braid algorithm} \label{sec:braid1}
	
	Let $\N$ be the set of nonnegative integers. 
	Using an algorithm to construct a toric $n$-mosaic for a $(p,q)$-torus knot, we will prove the following theorem.
	
	\begin{theorem} \label{thm:rapunzel}
		Let $K$ be a $(p,q)$-torus knot with $2\leq p<q$, and suppose that nonnegative integer solutions $(h,v)\in \N^2$ exist for the following system of inequalities:
		\begin{equation} \label{eq:rapunzel}
			\begin{cases}
				q-2(h+v+p)+4\geq 0,\\ 
				h + 3v \leq q - 3p + 4,\\ 
				-3h-v-p+q+4\geq 0,\\ 
				h \geq v.
			\end{cases}
		\end{equation}
		Choose a solution that minimizes $n:=q-(h+v)$. Then $m_T(K)\leq n$.
	\end{theorem}
	
	Optimal solutions to the system in (\ref{eq:rapunzel}) can be computed using an integer linear programming algorithm; we provide one in \texttt{toric.py} in \cite{github}. For example, an optimal solution when $p=2$ and $q\geq 3$ is $(h,v)=((q-1)/2,0)$. For $p=3$ and $7\leq q\leq 17$, some optimal solutions and their induced bounds are recorded in Table \ref{tab:one-braid-3}. 
	For $p=4$, optimal solutions and their induced bounds for all $q\geq 9$ are recorded in Table \ref{tab:one-braid-4}. This gives us the following. 
	
	\begin{cor}\label{cor:first-2-bound}
		If $K$ is a $(2,q)$-torus knot with $q\geq 3$, then $m_T(K)\leq (q+1)/2$.
	\end{cor}
	
	\begin{cor}
		Let $K$ be a $(4,q)$-torus knot. If $q=9$, then $m_T(K)\leq 8$. If $q\geq 11$, then $m_T(K)\leq 3+(q-1)/2$.
	\end{cor}
	
	Note that these bounds strictly improve upon the bound in Proposition \ref{naivebound}. In Section \ref{sec:braid2}, we further improve upon the bound in Corollary \ref{cor:first-2-bound} for all $q\geq 17$.
	
	
	\begin{table}[]
		\centering
		\begin{tabular}{c|cccccccc}
			$q$                                                               & $7$ & $8$ & $10$ & $11$ & $13$ & $14$ & $16$ & $17$ \\ \hline
			\begin{tabular}[c]{@{}c@{}}Optimal\\solution $(h,v)$\end{tabular} & $(2,0)$  & $(3,0)$ & $(3,0)$ & $(4,0)$ & $(4,1)$  & $(5,0)$  & $(5,2)$ & $(5,2)$ 
			\\ \hline
			\begin{tabular}[c]{@{}c@{}}Upper bound\\ on $m_T$\end{tabular} & $5$  & $5$ & $7$ & $7$ & $8$  & $9$  & $9$ & $10$ 
		\end{tabular}
		\caption{Optimal solutions $(h,v)$ to the system in (\ref{eq:rapunzel}) and induced bounds on the toric mosaic number from Theorem \ref{thm:rapunzel} for various $(3,q)$-torus knots.}
		\label{tab:one-braid-3}
	\end{table}
	
	\begin{table}[]
		\centering
		\begin{tabular}{c|cccc}
			$q$                                                                & $9$     & $11$    & $4k+1$, $k\geq 3$ & $4k-1$, $k\geq 4$ \\ \hline
			\begin{tabular}[c]{@{}c@{}}Optimal\\ solution $(h,v)$\end{tabular} & $(1,0)$ & $(3,0)$ & $(k+1,k-3)$       & $(k+1,k-4)$       \\ \hline
			\begin{tabular}[c]{@{}c@{}}Upper bound\\ on $m_T$\end{tabular}     & $8$     & $8$     & $3+(q-1)/2$       & $3+(q-1)/2$      
		\end{tabular}
		\caption{Optimal solutions $(h,v)$ to the system in (\ref{eq:rapunzel}) and induced bounds on the toric mosaic number from Theorem \ref{thm:rapunzel} for all $(4,q)$-torus knots with $q\geq 9$.}
		\label{tab:one-braid-4}
	\end{table}
	
	To prove Theorem \ref{thm:rapunzel}, we provide the following algorithm to construct a toric $n$-mosaic representing $K$. The resulting mosaic has $h$ sets of $p-1$ horizontally arranged crossings across $h$ rows and $v$ sets of $p-1$ vertically arranged crossings across $p+v-2$ rows. We call this algorithm the ``one-braid'' algorithm\footnote{For more fun, we encourage the reader to call it the ``Rapunzel'' algorithm instead.} due to the mosaic's appearance.
	
	\begin{defn}
		Consider a $(p,q)$-torus knot with $(h,v)\in\N^2$ satisfying the assumptions of Theorem \ref{thm:rapunzel}, and let $n:=q-(h+v)$. 
		The \textbf{one-braid algorithm}
		is the series of steps described below to construct a toric $n$-mosaic, as illustrated in Figure \ref{fig:one-braid-algo}. 
		\begin{figure}
			\centering
			\includegraphics[width=0.9\linewidth]{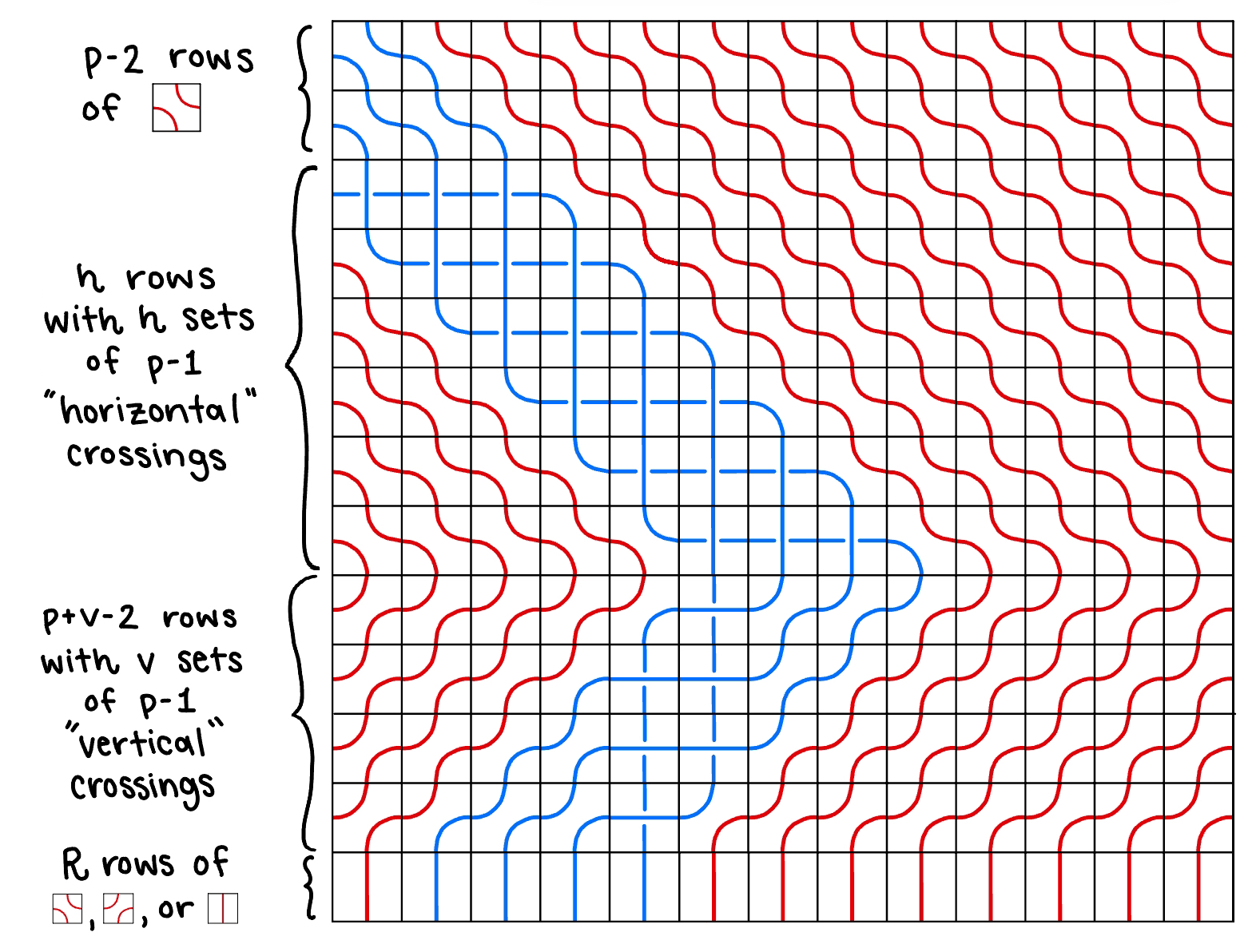}
			\caption{The ``one-braid'' algorithm. This example uses the optimal solution $(h,v)=(6,2)$ to construct a toric 13-mosaic representing the $(4,21)$-torus knot.}
			\label{fig:one-braid-algo}
		\end{figure}
	\end{defn}
	
	\begin{enumerate}
		\item On a blank $n\times n$ grid, label the rows $1,\dots,n$ from top to bottom. Let rows $1$ through $p-2$ be filled completely by $T_7$ tiles.
		\item For all integers $0\leq i \leq h-1$ and $1\leq j\leq p-1$, fill the tile in position $(p-1+i,i+j)$ with a $T_{10}$ tile. Note that this adds $p-1$ horizontally adjacent $T_{10}$ tiles to each of the $h$ rows labeled $p-1$ through $p-2+h$; these tiles are arranged in the manner represented in Figure \ref{fig:one-braid-algo}. Let all other tiles in these $h$ rows be $T_7$ tiles.
		\item For all integers $0\leq i \leq v-1$ and $1\leq j\leq p-1$, fill the tile in position $(p-2+h+i+j,h-i)$ with a $T_9$ tile. Note that this adds $p-1$ vertically adjacent $T_9$ tiles to each of the $v$ columns labeled $h-v+1$ through $h$; these tiles are arranged in the manner represented in Figure \ref{fig:one-braid-algo}. These tiles are contained in the $p+v-2$ rows labeled $p-1+h$ through $2p-4+h+v$; let all other tiles in these rows be $T_8$ tiles.
		\item Let $r:=p+v-h$, and let $R:=q-2(h+v+p)+4$. Note that the first three inequalities of (\ref{eq:rapunzel}) are equivalent to the inequalities $R\geq 0$, $R\geq r$, and $r\geq -R$, respectively. In particular, $R=|R|\geq |r|$.
		\item In steps (1) through (3), we filled $(p-2)+h+(p+v-2)=2p+h+v-4$ rows. Since $n=q-h-v$, the number of remaining blank rows is
		\[
		q-h-v-(2p+h+v-4)=q-2(h+v+p)+4 =R\geq 0. 
		\]
		Since $R\geq |r|$, we may consider the topmost $|r|$ of these blank rows. If $r\geq 0$, fill these rows with $T_8$ tiles; otherwise, fill them with $T_7$ tiles. Finally, fill the remaining $R-|r|$ rows at the bottom with $T_6$ tiles.
	\end{enumerate}
	
	By construction, the one-braid algorithm yields a suitably connected toric $n$-mosaic containing exactly $(h+v)(p-1)$ visible crossings. They occur between the same $p$ strands, namely those in the first $p-1$ tiles of the leftmost column.
	We will refer to these $p$ strands as the \emph{braid}. 
	
	To prove Theorem \ref{thm:rapunzel}, it will suffice to show that the toric mosaic produced by the one-braid algorithm represents the desired $(p,q)$-torus knot. To that end, we will show that this mosaic is topologically identical to an altered version of the classical representation of the $(p,q)$-torus knot, which we call an ``H-diagram'' due to its shape. See Figure \ref{fig:one-braid-H} for examples of H-diagrams when $p=2$.
	
	\begin{defn}
		Consider a $(p,q)$-torus knot $K$ with $(h,v)\in\N^2$ satisfying the assumptions of Theorem \ref{thm:rapunzel}. 
		We define the \textbf{H-diagram} of $K$ to be the diagram obtained by modifying the classical representation of $K$ (see Section \ref{sec:torus-knots}) through the steps described below. 
		\begin{figure}
			\centering
			\includegraphics[width=0.99\linewidth]{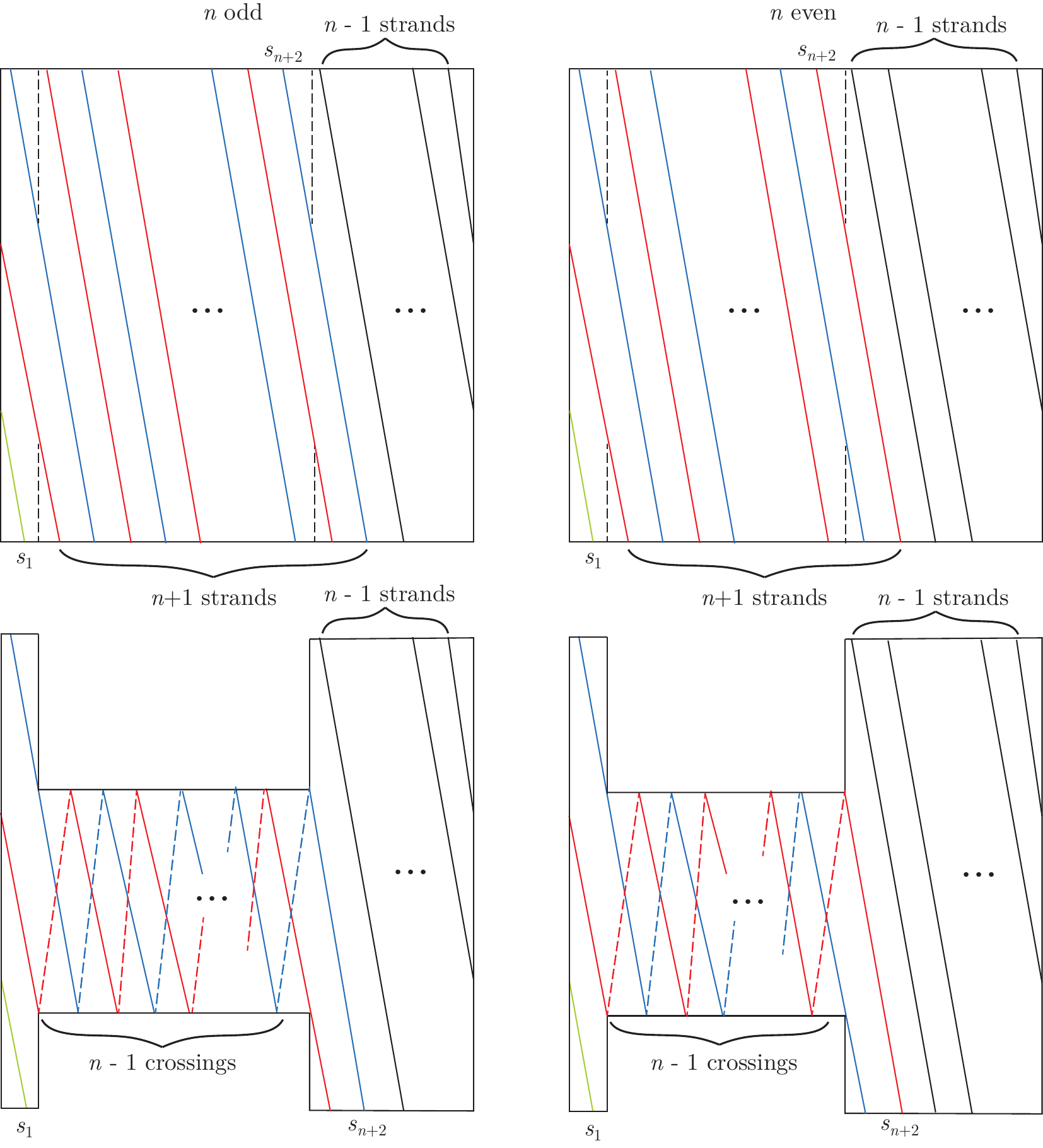}
			\caption{The one-cylinder ``H-diagram'' constructed from the classical representation of a $(2,q)$-torus knot. This example uses the optimal solution $(h,v)=((q-1)/2,0)$ to the system in (\ref{eq:rapunzel}), so $n=q-(h+v)=(q+1)/2$. Hence, $s_{h+v+p+1}=s_{n+2}$. Since $p=2$, there are 2 possible colors for $s_{n+2}$, determined by the parity of $n$.}
			\label{fig:one-braid-H}
		\end{figure}
	\end{defn}
	
	\begin{enumerate}
		\item Starting from the bottom-left corner of the diagram and moving toward the upper-right corner, label the strands of the diagram as $s_1,\dots,s_{p+q}$.
		\item Choose distinct colors labeled $1,\dots,p$. For all integers $1\leq i\leq h+v+p$, color $s_{i+1}$ with the color labeled $i\pmod{p}$.
		Choose another color labeled 0, and color only $s_1$ with it.
		\item Cut a vertical slit into the diagram starting at the midpoint of where $s_1$ and $s_2$ intersect the bottom edge of the diagram and ending at the unique point on this vertical line that intersects $s_2$. On this same vertical line, cut a second vertical slit from the top of the diagram stopping at the unique point on this line that intersects $s_p$. (This point exists since $s_1$ and $s_p$ are connected in the classical representation of $K$.)
		\item Cut two more vertical slits in a similar fashion, this time with $s_{h+v+1}$, $s_{h+v+2}$, and $s_{h+v+p+1}$ playing the roles of $s_1$, $s_2$, and $s_p$, respectively.
		\item Roll and glue the two rectangular regions of the diagram that the slits bound to obtain an ``H''-shaped diagram with a cylinder in the middle. Note that the slits on the top edge of the diagram have the same length, and the same goes for the bottom slits.
	\end{enumerate}
	
	By construction, all $h+v+p$ strands with nonzero colors wrap around the cylinder, so we may identify strands with the same colors as $p$ \emph{cylindrical strands}. Under this identification, the surface of the cylinder contains exactly one cylindrical strand of each nonzero color, and the only visible crossings occur on the surface of the cylinder between these $p$ strands. By construction, these crossings only involve the strands intersecting the bottom edge of the classical representation bounded by the two slits on the bottom, namely $s_2,\dots,s_{h+v+1}$. Each of these strands crosses the $p-1$ strands to its right on the classical representation.
	In other words, the H-diagram has exactly $(h+v)(p-1)$ visible crossings. These crossings all occur on the surface on the cylinder between the same $p$ cylindrical strands, each of which has a different nonzero color.

	
	With these facts established, we are now ready to prove Theorem \ref{thm:rapunzel}.
	
	\begin{proof}[Proof of Theorem \ref{thm:rapunzel}]
		
		We claim that there is a one-to-one correspondence between the strands of the H-diagram and the strands of the toric $n$-mosaic that preserves crossings and connection points.
		We begin by coloring the strands of the toric $n$-mosaic similarly to the strands of the H-diagram. Consider the first $p$ tiles in column 1. For all integers $1\leq i\leq p$, color the strand with a connection point at the left edge of the tile in position $(i,1)$ with the color labeled $i-1$. Color the remaining strand in the tile in position $(1,1)$ with the color labeled $p$. This makes the braid on our mosaic have $p$ colors.
		
		Now, consider the remaining colorless strands in column 1. By construction, none of these strands are a part of the braid. Make a cut containing these strands and glue it to the right side of the mosaic such that connection points are maintained. (This operation amounts to a rotation on the surface of the torus.) The resulting diagram, which we call an \emph{altered mosaic}, is represented at the bottom row of Figure \ref{fig:one-braid-finisher}. 
		
		With this identification, for any strand on the altered mosaic, one can consider its connection point at the top of the mosaic, traverse the strand downward, and stop when the strand connects to the bottom of the altered mosaic. The total signed number of columns one traverses in this manner, where rightward movement is positive, is given by counting how many rows contain $T_7$ tiles and how many contain $T_8$ tiles:
		\[
		(p-2)+h-(p+v+2)+r=h-v+r=h-v+(p+v-h)=p.
		\]
		If we label the connection points at the top of our (altered or original) mosaic $t_1,\dots,t_n$ going left to right, then for all integers $1\leq i\leq n$, the strand with connection point $t_i$ connects directly to the strand with top connection point $t_j$ and the strand with top connection point $t_k$, where $j=i+p\pmod{n}$ and $k=i-p\pmod{n}$.
	If we replace $n$ in this statement with $q$, then the resulting statement holds for the classical representation of $K$.
	Since our H-diagram has exactly $n$ connection points on its top edge (not counting the surface of the cylinder as an edge, of course), this statement also holds for our H-diagram.
	
	Therefore, there is a one-to-one correspondence between the strands of the altered mosaic and the strands of our H-diagram that preserves colors, order, crossings, and connection points between strands. This completes the proof.
\end{proof}

\begin{figure}
	\centering
	\includegraphics[width=0.99\linewidth]{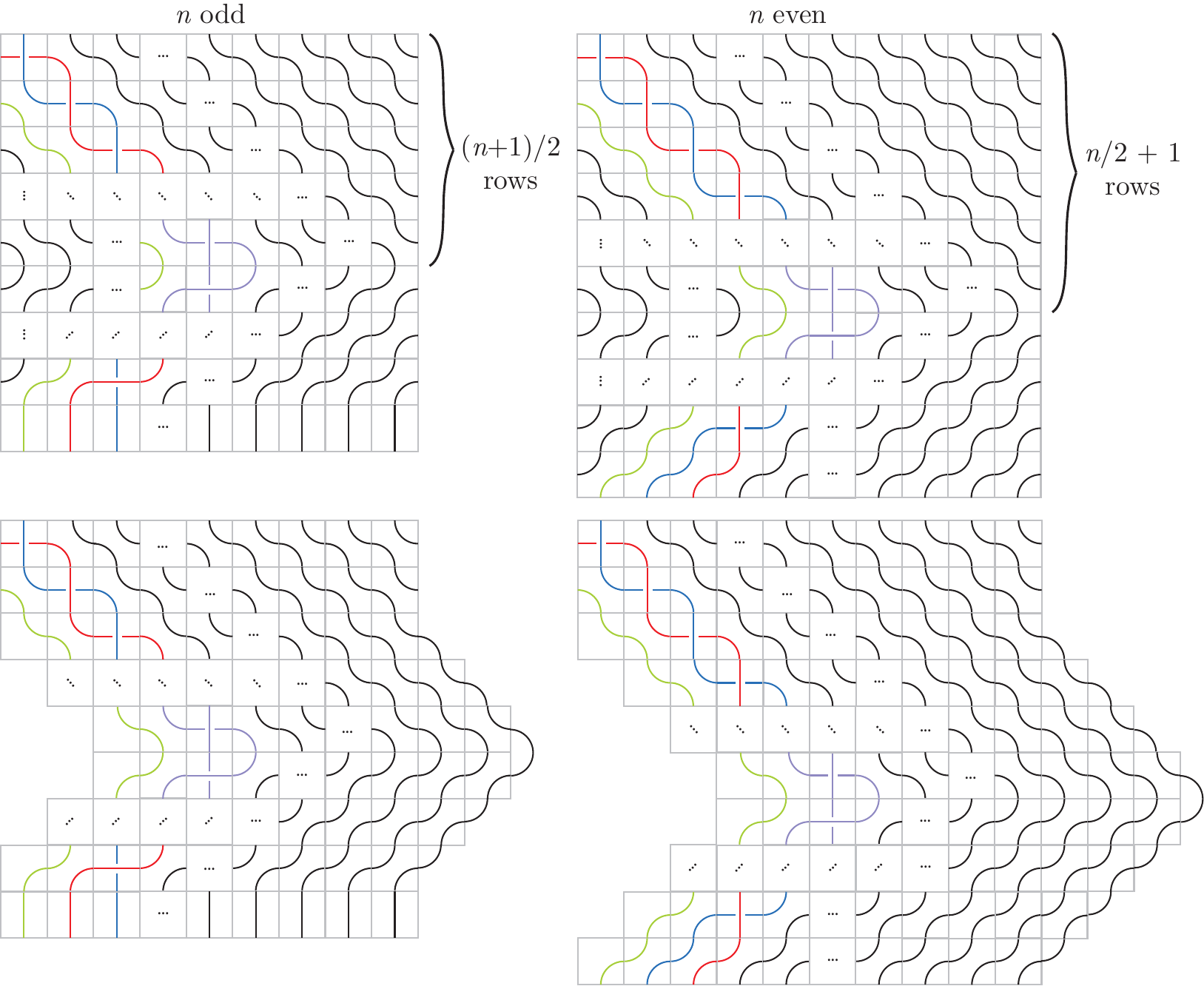}
	\caption{Exploiting the symmetry of the torus to move unbraided strands on the one-braid toric $n$-mosaic constructed for a $(2,q)$-torus knot. The resulting ``altered mosaic'' is naturally identified with the corresponding H-diagram (see Figure \ref{fig:one-braid-H}). In this image, purple denotes strands in the braid, which may be either red or blue.}
	\label{fig:one-braid-finisher}
\end{figure}

\begin{figure}
	\centering
	\includegraphics[width=0.9\linewidth]{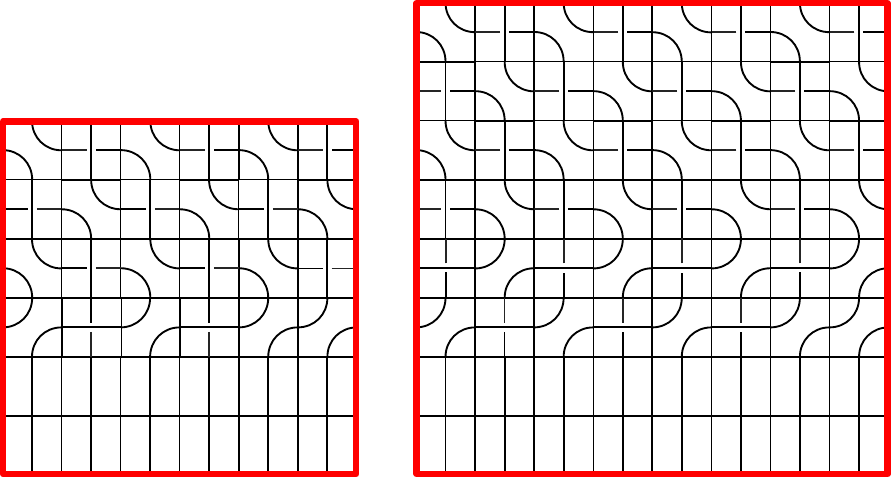}
	\caption{Toric $6$- and $8$-mosaics for the $(2,11)$- and $(2,23)$-torus knots constructed using the full-braid algorithm for $n=3,4$.}
	\label{fig:full-braid-1}
\end{figure}




\subsection{Full-braid algorithm for $(2,q)$-torus knots} \label{sec:braid2}

In the previous section, we constructed toric mosaics representing $(p,q)$-torus knots using a braided configuration of strands. In this section, we adapt this idea to introduce multiple braids on toric mosaics when $p=2$, giving us even more efficient toric mosaics representing $(2,q)$-torus knots. For this reason, we call our construction the ``full-braid algorithm.''

The full-braid algorithm is as follows. Given any integer $n\geq 3$, consider the odd integer $q':=2n^2-2n-1$. We create a toric $2n$-mosaic representing the $(2,q')$-torus knot through the following steps. For examples of this construction for $n=3,4,5$, see Figures \ref{fig:full-braid-1} and \ref{fig:full-braid-2}.

\begin{figure}
	\centering
	\includegraphics[width=0.6\linewidth]{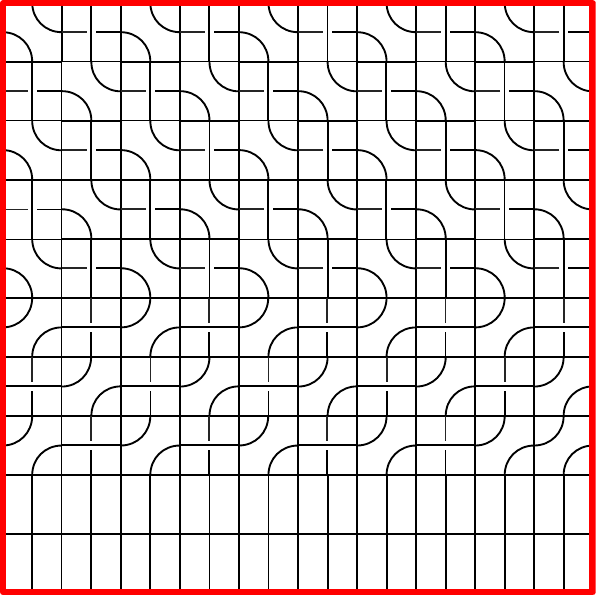}
	\caption{Toric $10$-mosaic for the $(2,39)$-torus knot constructed using the full-braid algorithm for $n=5$.}
	\label{fig:full-braid-2}
\end{figure}




\begin{enumerate}
	\item Given a blank $2n\times 2n$ grid, fill the bottom two rows with $T_6$ tiles.
	\item For each integer $1\leq i\leq n$, fill row $i$ in the following way. If $i$ is odd (resp.\ even), fill every odd-numbered (resp.\ even-numbered) tile in row $i$ with a $T_7$ tile. Then, fill all the remaining tiles in row $i$ with $T_{10}$ tiles.
	\item In row $2(n-1)$, fill every even-numbered tile with a $T_8$ tile. Then, fill the rightmost tile in row $2(n-1)$ with another $T_8$ tile, and fill all other tiles in row $2(n-1)$ with a $T_9$ tile.
	\item Consider the remaining rows $n+1,n+2,\dots, 2(n-1)-1$. For each integer $n+1\leq i\leq 2(n-1)-1$, fill row $i$ in the following way. If $i$ is odd (resp.\ if $i$ is even), fill every odd-numbered (resp.\ even-numbered) tile in row $i$ with a $T_9$ tile. Then, fill all the remaining tiles in row $i$ with $T_8$ tiles.
\end{enumerate}

Note that these steps produce a suitably connected toric mosaic with $q'$ visible crossings, namely, $n^2$ tiles of type $T_{10}$ from step (2) along with $n(n-2)-1$ tiles of type $T_9$ from steps (3) and (4). These crossing tiles are distributed across $n$ braids; one of the braids contains $2n-3$ visible crossings, while the other $n-1$ braids each contain $2n-2$ visible crossings. 

Therefore, iterating the cutting-and-gluing proof method in the previous section $n$ times shows that the toric mosaic described above represents a $(2,q')$-torus knot; we leave the details to the reader. This yields the following bound on the toric mosaic number of $(2,q')$-torus knots.

\begin{theorem}\label{thm:full-braid}
	For any integer $n\geq 3$, let $q':=2n^2-2n-1$. Then the toric mosaic number of the $(2,q')$-torus knot is at most $2n$.
\end{theorem}

Table \ref{tab:full-braid} displays the upper bounds obtained for all integers $3\leq n\leq 12$ in Theorem \ref{thm:full-braid}.

\begin{figure}
	\centering
	\includegraphics[width=0.9\linewidth]{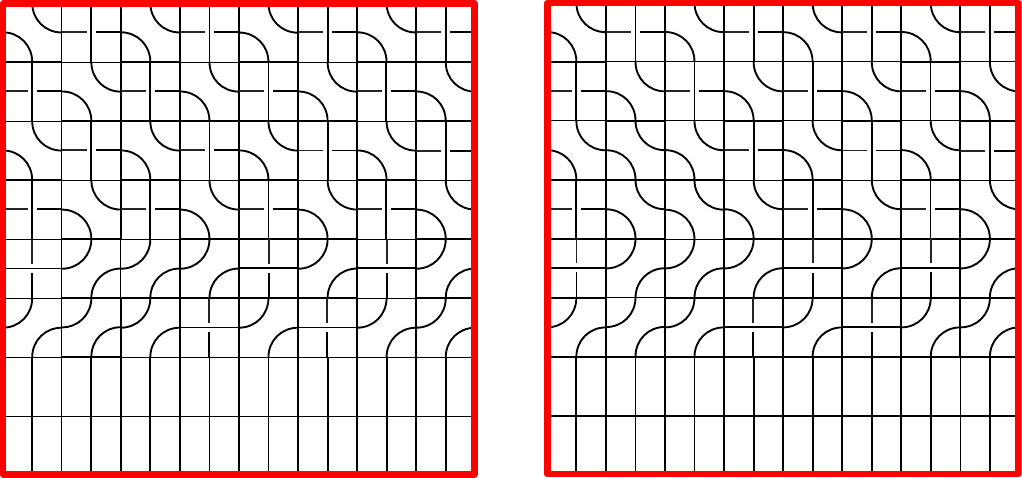}
	\caption{Toric $8$-mosaics for the $(2,21)$- and $(2,19)$-torus knots obtained by replacing crossing tiles in the toric mosaic for the $(2,23)$-torus knot in Figure \ref{fig:full-braid-1}.}
	\label{fig:full-braid-3}
\end{figure}

Thanks to the regular structure of $(2,q')$-torus knots, we can adapt the toric mosaic constructed above for a $(2,q')$-torus knot into a toric mosaic representing \emph{any} $(2,q)$-torus knot such that $3\leq q\leq q'$. This is accomplished by removing $q'-q$ consecutive crossings in a braid. In particular, we replace $q'-q$ consecutive crossing tiles (of types $T_9$ and $T_{10}$, respectively) with non-crossing double-strand tiles (of types $T_8$ and $T_7$, respectively); see Figure \ref{fig:full-braid-3} for examples. This observation yields the following corollary.

\begin{cor}\label{cor:full-braid}
	For any integer $n\geq 3$, and for any odd integer $q$ such that \[3\leq q \leq 2n^2-2n-1,\] the toric mosaic number of the $(2,q)$-torus knot is at most $2n$.
\end{cor}

Therefore, Table \ref{tab:full-braid} shows that the bound in Corollary \ref{cor:full-braid} is strictly stronger than the bound in Corollary \ref{cor:first-2-bound} for all $(2,q)$-torus knots with $q\geq 17$.

\begin{table}[]
	\centering
	\begin{tabular}{c|cccccccccc}
		$q$                                                               & $11$ & $23$ & $39$ & $59$ & $83$ & $111$ & $143$ & $179$ & $219$ & $263$ \\ \hline
		\begin{tabular}[c]{@{}c@{}}Upper bound\\ on $m_T$\end{tabular} & $6$  & $8$  & $10$ & $12$ & $14$ & $16$  & $18$  & $20$ & $22$ & $24$ 
	\end{tabular}
	\caption{Upper bounds on the toric mosaic number of various $(2,q)$-torus knots obtained from the full-braid algorithm in Theorem \ref{thm:full-braid}.}
	\label{tab:full-braid}
\end{table}

\section{Computer search} \label{sec:census}

In addition to algorithmic constructions, bounds on toric mosaic numbers can be obtained via a computer search. Our search was conducted using two programs: \texttt{mosaic-gen}, a \texttt{rust} program that produces a list of all suitably connected $n$-mosaics for a given $n$, and \texttt{toric.py}, a \texttt{python} program that catalogs those mosaics using the HOMFLY-PT polynomial. 

Full implementations and documentation for these programs can be found at \cite{github}. Toric mosaics are represented as base-11 numbers, obtained by flattening a mosaic's associated matrix as below.
\begin{equation*}
	\vcenter{\hbox{\includegraphics[scale=0.3]{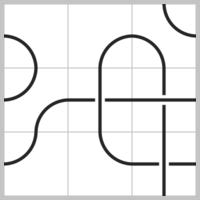}}} \Leftrightarrow \begin{bmatrix} 1 & 2 & 7 \\ 8 & 9 & 10 \\ 4 & 3 & 9\end{bmatrix} \Leftrightarrow 12789a439
\end{equation*}
A list of all suitably connected toric $n$-mosaics is then obtained by iterating over certain base-11 numbers of length $n^2$, and noting those mosaics that are suitably connected. 

From this list, low-crossing number knots can be readily identified using the HOMFLY-PT polynomial. Mosaics are first ``traced'' to produce a machine-readable planar diagram code, and discarded if they are links. If the mosaic is a knot, \texttt{Sagemath}'s \texttt{Links} package is used to calculate the knot's HOMFLY-PT polynomial, which is checked against a pre-computed list to identify the knot.

\subsection{Results}
Using these programs, we produced a complete census of those knots realizable on a toric 3-mosaic, and a partial census of those knots realizable on a toric 4-mosaic. Notable results are included below, and complete censuses are available in Appendix \ref{apd:censuses}. 

\begin{theorem}\label{thm:alg-1}
	There are exactly 9 prime knots realizable on a toric 3-mosaic, namely the unknot, $3_1$, $4_1$, $5_1$, $5_2$, $7_1$, $8_{19}$, $10_{124}$, $10_{139}$, and $10_{145}$.
\end{theorem}

\begin{theorem}\label{thm:alg-2}
	Every knot of crossing number less than or equal to 9 may be realized on a toric 4-mosaic.
\end{theorem}


\section{Further questions} \label{sec:questions}

For all odd $q\geq 17$, the full-braid algorithm in Section \ref{sec:braid2} provides sharper bounds on the toric mosaic numbers of $(2,q)$-torus knots than the one-braid algorithm in Section \ref{sec:braid1}. For this reason, it would be desirable to generalize the full-braid algorithm to all torus knots:

\begin{question}
	Can the full-braid algorithm in Section \ref{sec:braid2} be adapted for $(p,q)$-torus knots with $q>p\geq 3$?
\end{question}

Our findings in Section \ref{sec:census} suggest that the inequality in Observation \ref{initial-obs} is strict even for hyperbolic knots and satellite knots, leading us to ask the following:

\begin{question}
	Is it true that $m_T(K)<m(K)$ for all knots $K$?
\end{question}

By identifying the edges of a mosaic in different ways, we can also embed knot mosaics into surfaces like Klein bottles, projective planes, cylinders, and M\"obius strips. For each of these surfaces, we can define an analogue of the toric mosaic number, leading us to ask the following question:
\begin{question}
	When embedding knot mosaics into surfaces other than tori, what is the corresponding mosaic number of a given knot?
\end{question}

Furthermore, other embeddings of the torus could be considered. We ask:

\begin{question}
	How does the toric mosaic number of a knot change if a non-standard embedding is chosen?
\end{question}

\begin{question}
	What is the minimum toric mosaic number of knot when taken over all possible embeddings of the torus? 
\end{question}

Finally, toric mosaics have a natural relationship to virtual knot mosaics, first introduced in  \cite{virtualHenrich}. In particular, a virtual knot can be viewed as a knot on a surface, where the induced crossings correspond to virtual crossings. So, any hidden crossings induced by the toric mosaic can be viewed as “virtual.” Among others, some questions one could ask are:

\begin{question}
	Are there classes of virtual knots for which the toric mosaic number can be computed? 
\end{question}
\begin{question}
	Are there upper bounds on the toric mosaic number for infinite classes of virtual knots? 
\end{question}




\section*{Acknowledgments}
We would like to thank Allison Henrich for pointing us toward the literature on virtual knot mosaics. We would also like to thank Leif Schaumann for providing graphical templates and inspiring the full-braid algorithm in Section \ref{sec:braid2}.
We would also like to thank Kate Kearney, Lisa Traynor, and Peyton Wood for discussing toric link mosaics with us during the 2024 UnKnot V conference hosted at Seattle University.

This research was done at Moravian University as part of the Research Challenges of Computational Methods in Discrete Mathematics REU; it was funded by the National Science Foundation (MPS-2150299). 

\bibliography{references}

\pagebreak

\begin{center}
	Author contact information: \\
	\vskip 1cm
	Kendall Heiney\\
	Department of Mathematics, Cedar Crest College\\
	100 College Dr., Allentown, PA 18104, USA \\
	\href{mailto:kmheiney@cedarcrest.edu}{\tt kmHeiney@cedarcrest.edu}\\
	\ \\
	
	Margaret Kipe\\
	Department of Mathematics, University of Pittsburgh\\
	4200 Fifth Ave., Pittsburgh, PA 15260, USA\\
	\href{mailto:marge.k@pitt.edu}{\tt marge.k@pitt.edu}\\
	\ \\
	
	Samantha Pezzimenti\\
	Department of Mathematics, Penn State Brandywine\\
	25 Yearsley Mill Road, Media, PA 19063, USA\\
	\href{mailto:spezzimenti@psu.edu}{\tt spezzimenti@psu.edu}\\
	\ \\
	
	Kaelyn Pontes\\
	Department of Mathematics, Hastings College\\
	710 Turner Ave., Hastings, NE 68901, USA\\
	\href{mailto:kaelyn.pontes@hastings.edu}{\tt kaelyn.pontes@hastings.edu}\\
	\ \\
	
	L\d\uhorn c Ta\\
	Department of Mathematics, Yale University\\
	219 Prospect St., New Haven, CT 06511, USA\\
	\href{mailto:luc.ta@yale.edu}{\tt luc.ta@yale.edu}\\
	\ \\
	
\end{center}

\appendix

\twocolumn[
\section{Census of Toric Mosaic Numbers} \label{apd:censuses}
In this section, we present bounds on toric mosaic number for all knots of crossing number 10 or lower. For knots with a computationally verified minimal toric mosaic, a base-11 representation is included as in Section \ref{sec:census}. Images of these mosaics can be generated from these numbers via \texttt{toric.py}, available at \cite{github}. Knots $10_{163}$--$10_{166}$ in Rolfsen's table have been renumbered as $10_{162}$--$10_{165}$ to account for the Perko pair.
\vspace{12pt}
]
\input{appendices/census}

\end{document}

%% file: appendices/census.tex
\tablecaption{Bounds on toric mosaic number for knots through 10 crossings.}

\setlength{\tabcolsep}{4pt}
\tablehead{
\hline
\cellcolor{black!15}$K$   & \cellcolor{black!15}$m_T(K)$  & \cellcolor{black!15}Minimal Mosaic \\
\hline
}
\tabletail{\hline}
\begin{supertabular}[h]{|c c r|}
    $0_1$ & $1$ & \texttt{7} \\
    $3_1$ & $2$ & \texttt{7779} \\
    $4_1$ & $3$ & \texttt{12789a439} \\
    $5_1$ & $3$ & \texttt{294942429} \\
    $5_2$ & $4$ & \texttt{7778998a9} \\
    $6_1$ & $4$ & \texttt{0066127789aa4399} \\
    $6_2$ & $4$ & \texttt{0066127789a94399} \\
    $6_3$ & $4$ & \texttt{0066127779a9439a} \\
    $7_1$ & $3$ & \texttt{88889989a} \\
    $7_2$ & $4$ & \texttt{0316129797a93946} \\
    $7_3$ & $4$ & \texttt{0284289198a84399} \\
    $7_4$ & $4$ & \texttt{02742a9179a9439a} \\
    $7_5$ & $4$ & \texttt{028428a1979a4399} \\
    $7_6$ & $4$ & \texttt{028428919989439a} \\
    $7_7$ & $4$ & \texttt{027429a18a9a43a9} \\
    $8_1$ & $4$ & \texttt{0391127999a73994} \\
    $8_2$ & $4$ & \texttt{029429719a984389} \\
    $8_3$ & $4$ & \texttt{127779a9979a4397} \\
    $8_4$ & $4$ & \texttt{037112799a9739a4} \\
    $8_5$ & $4$ & \texttt{127789a98a99439a} \\
    $8_6$ & $4$ & \texttt{127789a98999439a} \\
    $8_7$ & $4$ & \texttt{0371129797a93994} \\
    $8_8$ & $4$ & \texttt{1277787979a9439a} \\
    $8_9$ & $4$ & \texttt{127779a9a978439a} \\
    $8_{10}$ & $4$ & \texttt{1277787979a9439a} \\
    $8_{11}$ & $4$ & \texttt{0371129999a73a94} \\
    $8_{12}$ & $4$ & \texttt{127779a8a799439a} \\
    $8_{13}$ & $4$ & \texttt{028429a1989a43a9} \\
    $8_{14}$ & $4$ & \texttt{127779a9979743a9} \\
    $8_{15}$ & $4$ & \texttt{0371129979a93a94} \\
    $8_{16}$ & $4$ & \texttt{1277789779a9439a} \\
    $8_{17}$ & $4$ & \texttt{127779a97799439a} \\
    $8_{18}$ & $4$ & \texttt{1289989aaa7743a9} \\
    $8_{19}$ & $3$ & \texttt{888888899} \\
    $8_{20}$ & $4$ & \texttt{02842881899a43a9} \\
    $8_{21}$ & $4$ & \texttt{028428918a9943a9} \\
    $9_1$ & $4$ & \texttt{7778a989787a89a9} \\
    $9_2$ & $4$ & \texttt{77797aa989aa8997} \\
    $9_3$ & $4$ & \texttt{1297898a98a84399} \\
    $9_4$ & $4$ & \texttt{127799a7997a4397} \\
    $9_5$ & $4$ & \texttt{039112799a9739a4} \\
    $9_6$ & $4$ & \texttt{1287899a98a84399} \\
    $9_7$ & $4$ & \texttt{7778a89997a7797a} \\
    $9_8$ & $4$ & \texttt{1289889a99a74398} \\
    $9_9$ & $4$ & \texttt{7778a979789a97a9} \\
    $9_{10}$ & $4$ & \texttt{7779897a89a97a97} \\
    $9_{11}$ & $4$ & \texttt{127979a9979a4378} \\
    $9_{12}$ & $4$ & \texttt{127779a9997a4397} \\
    $9_{13}$ & $4$ & \texttt{127979a9a978439a} \\
    $9_{14}$ & $4$ & \texttt{7778a889889a9979} \\
    $9_{15}$ & $4$ & \texttt{1289899a9a9743a8} \\
    $9_{16}$ & $4$ & \texttt{7778a899798aa8a9} \\
    $9_{17}$ & $4$ & \texttt{1288a89999a74398} \\
    $9_{18}$ & $4$ & \texttt{127979a8a979439a} \\
    $9_{19}$ & $4$ & \texttt{127979a997984379} \\
    $9_{20}$ & $4$ & \texttt{127789a99a9843a9} \\
    $9_{21}$ & $4$ & \texttt{128a89a99a8843a9} \\
    $9_{22}$ & $4$ & \texttt{127798a99a9843a9} \\
    $9_{23}$ & $4$ & \texttt{7778a9799a7778a9} \\
    $9_{24}$ & $4$ & \texttt{127998a7979a43a9} \\
    $9_{25}$ & $4$ & \texttt{127799a8a799439a} \\
    $9_{26}$ & $4$ & \texttt{127799a7979a4379} \\
    $9_{27}$ & $4$ & \texttt{127979a9a798439a} \\
    $9_{28}$ & $4$ & \texttt{127799a7979a43a9} \\
    $9_{29}$ & $4$ & \texttt{7778a9789a977979} \\
    $9_{30}$ & $4$ & \texttt{12878a9a89a94399} \\
    $9_{31}$ & $4$ & \texttt{127779a9979a43a9} \\
    $9_{32}$ & $4$ & \texttt{127779a9979a4379} \\
    $9_{33}$ & $4$ & \texttt{7778a8a98a9a97a9} \\
    $9_{34}$ & $4$ & \texttt{77aaa97789799a98} \\
    $9_{35}$ & $4$ & \texttt{7779899a8a998a97} \\
    $9_{36}$ & $4$ & \texttt{127979a997974379} \\
    $9_{37}$ & $4$ & \texttt{127779a9979a4397} \\
    $9_{38}$ & $4$ & \texttt{7779899889a98a9a} \\
    $9_{39}$ & $4$ & \texttt{78a98978989979a9} \\
    $9_{40}$ & $4$ & \texttt{7a987a89989a88a9} \\
    $9_{41}$ & $4$ & \texttt{1289989889a94398} \\
    $9_{42}$ & $4$ & \texttt{02842891899a43a9} \\
    $9_{43}$ & $4$ & \texttt{0284289197a94399} \\
    $9_{44}$ & $4$ & \texttt{0371127979a93994} \\
    $9_{45}$ & $4$ & \texttt{0284299199a7439a} \\
    $9_{46}$ & $4$ & \texttt{02842891989a43a9} \\
    $9_{47}$ & $4$ & \texttt{02842891989a43a9} \\
    $9_{48}$ & $4$ & \texttt{028429a18a9a43a9} \\
    $9_{49}$ & $4$ & \texttt{02842891989a43a9} \\
    $10_{1}$ & $4$ & \texttt{78a8979a9978a997} \\
    $10_{2}$ & $4$ & \texttt{787a88a97a98a989} \\
    $10_{3}$ & $4$ & \texttt{127999a7997a4397} \\
    $10_{4}$ & $4$ & \texttt{78a8779aa9789a97} \\
    $10_{5}$ & $4$ & \texttt{787979a9997a7997} \\
    $10_{6}$ & $4$ & \texttt{787a89a989989a89} \\
    $10_{7}$ & $4$ & \texttt{127979a9997a4397} \\
    $10_{8}$ & $4$ & \texttt{7a99798a88a989a8} \\
    $10_{9}$ & $4$ & \texttt{8788a998899a99a8} \\
    $10_{10}$ & $4$ & \texttt{787a88a98a98a989} \\
    $10_{11}$ & $4$ & \texttt{787979a99a78a997} \\
    $10_{12}$ & $4$ & \texttt{787799a7779aa997} \\
    $10_{13}$ & $4$ & \texttt{787979a99777a997} \\
    $10_{14}$ & $4$ & \texttt{787a89a989989a98} \\
    $10_{15}$ & $4$ & \texttt{7779899a89797a97} \\
    $10_{16}$ & $4$ & \texttt{7a98998a97a88989} \\
    $10_{17}$ & $\geq4$ & n/a\\
    $10_{18}$ & $4$ & \texttt{78a999788a89a998} \\
    $10_{19}$ & $4$ & \texttt{887989a9aa88989a} \\
    $10_{20}$ & $4$ & \texttt{787979a99978a997} \\
    $10_{21}$ & $4$ & \texttt{787a89a9a7989a89} \\
    $10_{22}$ & $4$ & \texttt{7778a97978a9979a} \\
    $10_{23}$ & $4$ & \texttt{787979a97798a979} \\
    $10_{24}$ & $4$ & \texttt{787979a99797799a} \\
    $10_{25}$ & $4$ & \texttt{787979a97a98a979} \\
    $10_{26}$ & $4$ & \texttt{7779899a89a97a97} \\
    $10_{27}$ & $4$ & \texttt{7778a9797a9797a9} \\
    $10_{28}$ & $4$ & \texttt{787a89a9989879a9} \\
    $10_{29}$ & $4$ & \texttt{787a89a98a9879a9} \\
    $10_{30}$ & $4$ & \texttt{7779899a9a79a897} \\
    $10_{31}$ & $4$ & \texttt{7a9779897798a799} \\
    $10_{32}$ & $4$ & \texttt{7779899a98a97a97} \\
    $10_{33}$ & $4$ & \texttt{78a9977a79a8a979} \\
    $10_{34}$ & $4$ & \texttt{787a88a99798a989} \\
    $10_{35}$ & $4$ & \texttt{7a9797a8a889a998} \\
    $10_{36}$ & $4$ & \texttt{787979a8779aa979} \\
    $10_{37}$ & $4$ & \texttt{787979a87a9aa997} \\
    $10_{38}$ & $4$ & \texttt{7778a9799a9778a9} \\
    $10_{39}$ & $4$ & \texttt{787a89a9a7989a98} \\
    $10_{40}$ & $4$ & \texttt{787979a9a7987a9a} \\
    $10_{41}$ & $4$ & \texttt{7a8aa989989888a9} \\
    $10_{42}$ & $4$ & \texttt{7778a8a99a9779a9} \\
    $10_{43}$ & $4$ &\texttt{78a97a87799aa979} \\
    $10_{44}$ & $4$ & \texttt{7778a8a9989a89a9} \\
    $10_{45}$ & $4$ & \texttt{77797a9aa9799897} \\
    $10_{46}$ & $4$ & \texttt{787a88a99898a989} \\
    $10_{47}$ & $4$ & \texttt{787979a99977a997} \\
    $10_{48}$ & $4$ & \texttt{787998a79a98a997} \\
    $10_{49}$ & $4$ & \texttt{78a97a9a7988a989} \\
    $10_{50}$ & $4$ & \texttt{78a9897a9898a989} \\
    $10_{51}$ & $4$ & \texttt{78a9979a8a77a997} \\
    $10_{52}$ & $\geq4$ & n/a\\
    $10_{53}$ & $4$ & \texttt{7779899a89a79a97} \\
    $10_{54}$ & $4$ & \texttt{78a9977879799a97} \\
    $10_{55}$ & $4$ & \texttt{77798a9a89979a97} \\
    $10_{56}$ & $4$ & \texttt{787979a97a98a979} \\
    $10_{57}$ & $4$ & \texttt{78a97a97978aa979} \\
    $10_{58}$ & $4$ & \texttt{7778a979979779a9} \\
    $10_{59}$ & $4$ & \texttt{7778a899989aa8a9} \\
    $10_{60}$ & $4$ & \texttt{7778a8a9989a9979} \\
    $10_{61}$ & $4$ & \texttt{7778a989989889a9} \\
    $10_{62}$ & $4$ & \texttt{7778a9797aa9979a} \\
    $10_{63}$ & $4$ & \texttt{127979a9979a4397} \\
    $10_{64}$ & $4$ & \texttt{787799a79a98a997} \\
    $10_{65}$ & $4$ & \texttt{7778a9799a7797a9} \\
    $10_{66}$ & $4$ & \texttt{7a8aa9899888a989} \\
    $10_{67}$ & $4$ & \texttt{7778a979979797a9} \\
    $10_{68}$ & $4$ & \texttt{7a97a987899a98a7} \\
    $10_{69}$ & $4$ & \texttt{7778a89979a97a9a} \\
    $10_{70}$ & $4$ & \texttt{787998a7979a79a9} \\
    $10_{71}$ & $4$ & \texttt{7778a89999a77a9a} \\
    $10_{72}$ & $4$ & \texttt{78a98a989789a998} \\
    $10_{73}$ & $4$ & \texttt{78a97a899798a979} \\
    $10_{74}$ & $4$ & \texttt{787999a7979a79a8} \\
    $10_{75}$ & $4$ & \texttt{887a89a89a88a989} \\
    $10_{76}$ & $\geq4$ & n/a\\
    $10_{77}$ & $4$ & \texttt{7778a97999a7979a} \\
    $10_{78}$ & $4$ & \texttt{7a979789979889a9} \\
    $10_{79}$ & $4$ & \texttt{787979a99a88a997} \\
    $10_{80}$ & $4$ & \texttt{7779899889a9989a} \\
    $10_{81}$ & $\geq4$ & n/a\\
    $10_{82}$ & $4$ & \texttt{7a9797a88a89a998} \\
    $10_{83}$ & $4$ & \texttt{78a9979a8a79a977} \\
    $10_{84}$ & $4$ & \texttt{7a9779897798a979} \\
    $10_{85}$ & $4$ & \texttt{1279989779a9439a} \\
    $10_{86}$ & $\geq4$ & n/a\\
    $10_{87}$ & $4$ & \texttt{78a99778a9799a77} \\
    $10_{88}$ & $\geq4$ & n/a\\
    $10_{89}$ & $4$ & \texttt{7a9779a9a798897a} \\
    $10_{90}$ & $\geq4$ & n/a\\
    $10_{91}$ & $4$ & \texttt{7a977789999789a9} \\
    $10_{92}$ & $4$ & \texttt{887989a89a899998} \\
    $10_{93}$ & $4$ & \texttt{7a9799a889999877} \\
    $10_{94}$ & $4$ & \texttt{7a987989988a89a8} \\
    $10_{95}$ & $4$ & \texttt{7778a9797a9779a9} \\
    $10_{96}$ & $\geq4$ & n/a\\
    $10_{97}$ & $\geq4$ & n/a\\
    $10_{98}$ & $\geq4$ & n/a\\
    $10_{99}$ & $4$ & \texttt{7778a9798a9a79a9} \\
    $10_{100}$ & $4$ & \texttt{12779a9779a9439a} \\
    $10_{101}$ & $4$ & \texttt{78a9979aa7797a77} \\
    $10_{102}$ & $4$ & \texttt{78a9979a97797997} \\
    $10_{103}$ & $4$ & \texttt{787979a9a7987a9a} \\
    $10_{104}$ & $4$ & \texttt{787799a78a9aa997} \\
    $10_{105}$ & $\geq4$ & n/a\\
    $10_{106}$ & $4$ & \texttt{7778a979989a97a9} \\
    $10_{107}$ & $4$ & \texttt{887989a89a89998a} \\
    $10_{108}$ & $4$ & \texttt{7a987889989a8989} \\
    $10_{109}$ & $4$ & \texttt{7778a9798a9a97a9} \\
    $10_{110}$ & $\geq4$ & n/a\\
    $10_{111}$ & $4$ & \texttt{7a98989a88998989} \\
    $10_{112}$ & $4$ & \texttt{7879989779a99a98} \\
    $10_{113}$ & $4$ & \texttt{7a97798a97a8a979} \\
    $10_{114}$ & $\geq4$ & n/a\\
    $10_{115}$ & $\geq4$ & n/a\\
    $10_{116}$ & $4$ & \texttt{77798a9a79a99798} \\
    $10_{117}$ & $4$ & \texttt{887999a8a988989a} \\
    $10_{118}$ & $4$ & \texttt{78a99a98a9777a97} \\
    $10_{119}$ & $\geq4$ & n/a\\
    $10_{120}$ & $\geq4$ & n/a\\
    $10_{121}$ & $\geq4$ & n/a\\
    $10_{122}$ & $4$ & \texttt{1289989a89a94398} \\
    $10_{123}$ & $4$ & \texttt{78a99a98a9797a77} \\
    $10_{124}$ & $~~~3~~~$ & \texttt{888899998} \\
    $10_{125}$ & $4$ & \texttt{127789a998a84399} \\
    $10_{126}$ & $4$ & \texttt{12799897799a43a9} \\
    $10_{127}$ & $4$ & \texttt{127799a79a984399} \\
    $10_{128}$ & $4$ & \texttt{12778a9999894399} \\
    $10_{129}$ & $\geq4$ & n/a\\
    $10_{130}$ & $4$ & \texttt{127799a89a7943a9} \\
    $10_{131}$ & $4$ & \texttt{12779a9799a84379} \\
    $10_{132}$ & $3$ & \texttt{139913391} \\
    $10_{133}$ & $4$ & \texttt{0371129997a93994} \\
    $10_{134}$ & $4$ & \texttt{12779a9897a94399} \\
    $10_{135}$ & $4$ & \texttt{12779a9797a94399} \\
    $10_{136}$ & $4$ & \texttt{127779a8799a4399} \\
    $10_{137}$ & $4$ & \texttt{127779a8979a4399} \\
    $10_{138}$ & $4$ & \texttt{1287899a9a7943a9} \\
    $10_{139}$ & $3$ & \texttt{888899899} \\
    $10_{140}$ & $4$ & \texttt{0371129999793994} \\
    $10_{141}$ & $4$ & \texttt{127799a7799a4399} \\
    $10_{142}$ & $4$ & \texttt{12778a9997a94399} \\
    $10_{143}$ & $4$ & \texttt{127779a99a9743a9} \\
    $10_{144}$ & $4$ & \texttt{128789a9989a43a9} \\
    $10_{145}$ & $4$ & \texttt{02842991899a43a9} \\
    $10_{146}$ & $4$ & \texttt{12777a9aa7a9439a} \\
    $10_{147}$ & $4$ & \texttt{127798a98a9943a9} \\
    $10_{148}$ & $4$ & \texttt{127779a99a97439a} \\
    $10_{149}$ & $4$ & \texttt{127779a99a984399} \\
    $10_{150}$ & $4$ & \texttt{127789a98a9943a9} \\
    $10_{151}$ & $4$ & \texttt{12777a99a989439a} \\
    $10_{152}$ & $4$ & \texttt{03711299a9893994} \\
    $10_{153}$ & $4$ & \texttt{127789aa89794399} \\
    $10_{154}$ & $4$ & \texttt{0371129999893994} \\
    $10_{155}$ & $4$ & \texttt{127799a77a9a43a9} \\
    $10_{156}$ & $4$ & \texttt{1277789779a9439a} \\
    $10_{157}$ & $4$ & \texttt{7778a9787a9a7979} \\
    $10_{158}$ & $4$ & \texttt{12777a9a89a94399} \\
    $10_{159}$ & $4$ & \texttt{1288799aa8a9439a} \\
    $10_{160}$ & $4$ & \texttt{127a89a98a9943a8} \\
    $10_{161}$ & $4$ & \texttt{03711299799939a4} \\
    $10_{162}$ & $4$ & \texttt{12777a9a99794399} \\
    $10_{163}$ & $4$ & \texttt{1289989a9a7943a8} \\
    $10_{164}$ & $4$ & \texttt{1288989a89a84399} \\
    $10_{165}$ & $4$ & \texttt{1287889a89a94399} \\
\end{supertabular}